\documentclass[12pt,reqno]{amsart}

\usepackage[color]{showkeys}     % refs and labels
\definecolor{refkey}{gray}{.5}   % graylevel for refs
\definecolor{labelkey}{gray}{.5} % graylevel for labels

\usepackage{epsfig}
\usepackage{amsmath, amsthm, amsfonts, amssymb, mathrsfs}
\usepackage{graphicx}

\numberwithin{equation}{section}

\newcommand{\R}{{\mathbb R}}
\newcommand{\C}{{\mathbb C}}

\newcommand{\Z}{{\mathbb Z}}

\newcommand{\K}{{\mathbb K}}

\newtheorem{theo}{{\sc \bf Theorem}}[section]
\newtheorem{cor}[theo]{{\sc \bf Corollary}}
\newtheorem{lem}[theo]{{\sc \bf Lemma}}
\newtheorem{prop}[theo]{{\sc \bf Proposition}}

\begin{document}

\title{Derivations and Reflection Positivity on the Quantum cylinder}

\author[Klimek]{Slawomir Klimek}
\address{Department of Mathematical Sciences,
Indiana University-Purdue University Indianapolis,
402 N. Blackford St., Indianapolis, IN 46202, U.S.A.}
\email{sklimek@math.iupui.edu}

\author[McBride]{Matt McBride}
\address{Department of Mathematics and Statistics,
Mississippi State University,
175 President's Cir., Mississippi State, MS 39762, U.S.A.}
\email{mmcbride@math.msstate.edu}

\thanks{}

\date{\today}

\begin{abstract}
We describe the general structure of unbounded derivations in the quantum cylinder. We prove a noncommutative analog of reflection positivity for Laplace-type operators in a noncommutative cylinder following the ideas of Jaffe and Ritter proof of reflection positivity for Laplace operators on manifolds equipped with a reflection. 
\end{abstract}

\maketitle
\section{Introduction}

Part of this work is a continuation of the program started in \cite{KMRSW1} and \cite{KMR} of studying unbounded derivations in quantum domains,
their implementations, and possible spectral triples associated to them.
Another part of this work was inspired by a Glimm and Jaffe note \cite{GJ} on reflection positivity for the Laplace operator in $\R^n$. Additionally, we were influenced by the Jaffe and Ritter paper \cite{JR}, which considered reflection positivity for Laplace operators on manifolds equipped with a reflection. 

Reflection positivity in the Euclidean space is the following remarkable inequality in $L^2(\R^n)$:
\begin{equation*}
\langle \Theta f, \left(-\Delta+ m^2\right)^{-1}f\rangle \ge 0 
\end{equation*}
for all $f\in \mathcal{H}^+ = \{f\in L^2(\R^n) : f(x_1,\ldots,x_n) = 0,\textrm{ for }x_1<0\}$. Here $\Delta$ is the Laplace operator in $\R^n$, $m$ is a constant and $\Theta: L^2(\R^n)\to  L^2(\R^n)$ is the reflection in the first coordinate: 
$$\Theta f(x_1,\ldots,x_{n})=f(-x_1,\ldots,x_{n}).$$
This inequality is a key step in proving the Reflection Positivity Axiom of Osterwalder-Schrader for the free field, see \cite{GJbook}. Reflection positivity has been generalized in many directions, of particular interest for this note is already mentioned manifold generalization in \cite{JR}.

A natural question is then if such ideas can be extended to noncommutative geometry to include examples of Laplace-type operators on noncommutative manifolds. One of the simplest possibilities, studied in detail in this paper, is a quantum cylinder which classically has a natural reflection through the middle. 

The noncommutative cylinder (quantum annulus) was constructed in \cite{KL3} and further studied in \cite{KMRS} and \cite{KMR}. It has a natural rotational symmetry as well as a reflection, as will be shown later, and it also has an analog of the Lebesgue measure. To define a class of interesting, reflection invariant, Laplace-type operators in the corresponding Hilbert space we use rotationally covariant unbounded derivations on the quantum cylinder that we studied in \cite{KMRSW1} and \cite{KMR}. With proper choices we indeed get a reflection positivity for such analogs of Laplace operators with the proof following closely the ideas in \cite{JR}. 

When working with noncommutative algebras, we tried to make our constructions as geometrical as possible. Every step was carefully motivated by the corresponding classical geometry concepts and their noncommutative versions.

The paper is organized as follows.  In Section 2 we review the quantum cylinder and discuss its geometric  Hilbert space constructed from an invariant weight playing the role of the Lebesgue measure.  We also discuss the reflection operator necessary for the reflection positivity for the Laplace-type operators we consider.  In Section 3 we classify all unbounded derivations on the quantum cylinder that arise from the dense subalgebra considered in Section 2.  Finally in Section 4 we show how to implement those derivations which are invariant and covariant and create Laplace-type operators from those implementations.  Moreover we prove the reflection positivity for such a class of Laplace-type operators.

\section{Quantum cylinder}

In this section we describe the structure and the geometry of the noncommutative cylinder, which is defined here as a concrete  C$^*-$algebra of operators in $\ell^2(\Z)$. To introduce it we need some notation.

Let $\{E_k\}$ be the canonical basis for $\ell^2(\Z)$ and let $U$ be the bilateral shift, i.e. 
$$UE_k = E_{k+1}.$$  
We use the diagonal label operator: 
$$\K E_{k} = kE_{k},$$ 
so that, for a bounded function $a: \Z \to\C$, we can write: 
$$a(\K) E_{k} = a(k)E_{k}.$$  
In a sense the operators $U$, $\K$ are noncommutative polar coordinates. We have the following crucial commutation relation for a diagonal operator $a(\K)$:\begin{equation}\label{CommRel}
a(\K)\,U = Ua(\K +1).
\end{equation}

Let $c(\Z)$ be the space of convergent sequences, and consider the abelian algebra: 
$$A_{diag} = \left\{a(\K)  : \{a(k)\}\in c(\Z) \right\}.$$   
We define the quantum cylinder as the C$^*-$algebra generated by $U$ and the above diagonal operators:
$$A=C^*(A_{diag}, U).$$
Because of formula \eqref{CommRel}, we can view the quantum cylinder as the group crossed product of $A_{diag}$ with $\Z$ acting on $A_{diag}$  via shifts (translation by $n\in\Z$), that is: \begin{equation*}
A = A_{diag} \rtimes_{shift}\Z.
\end{equation*}
Also, we see that $(A,  A_{diag})$ is a Cartan pair \cite{R}.

Alternatively, the quantum cylinder can be described as a singly generated C$^*-$algebra. Consider the following special weighted bilateral shift:\begin{equation*}
U_rE_k = \left\{
\begin{aligned}
rE_{k+1} & \quad k<0 \\
E_{k+1} & \quad k\ge 0,
\end{aligned}\right.
\end{equation*}
$0<r<1$.
We have that $A = C^*(U_r)$, the unital C$^*-$algebra generated by $U_r$, \cite{KMRS}. Additional arguments in references  \cite{KL3}, \cite{KM1}, and \cite{KMRS} further justify that this noncommutative $C^*-$algebra is an appropriate choice to be called the quantum cylinder.   Its structure can be described by the following short exact sequence:\begin{equation*}
0\longrightarrow \mathcal{K} \longrightarrow A  \longrightarrow C(S^1)\oplus C(S^1) \longrightarrow 0,
\end{equation*}
where $\mathcal{K}$ is the ideal of compact operators in $\ell^2(\Z)$ which, in fact, is the commutator ideal of the algebra $A$.   

As in \cite{KMR}, we call a function $a:\Z \to\C$ {\it eventually constant}, if there exists a natural number $k_0$ such that $a(k)$ are constants for $|k|\ge k_0$.  The constants are denoted by $a(\pm\infty)$. The set of all such functions will be denoted by $c_{00}^\pm(\Z)$. 

Let $Pol(U_r, U_r^*)$ be the set of all polynomials in $U_r$ and $U_r^*$. Alternatively, we have a natural dense algebra of ``algebraic" elements of $A$ defined as:\begin{equation*}
\mathcal{A} = \left\{a = \sum_{n}U^n a_n(\K) \ : \ a(k)\in c_{00}^\pm(\Z),\textrm{ finite sum}\right\}.
\end{equation*}
It was proved in \cite{KMR} that $\mathcal{A} = Pol(U_r, U_r^*)$, which allows us to work with the unitary $U$ and diagonal operators, as dealing with $U_r$ is more difficult. In fact, we have the following useful identification.

\begin{prop}\label{useful_gens}
$\mathcal{A}$ is equal to the algebra of polynomial generated by $U$, $U^{-1}$, $P_{\ge0}$, and $P_0$, where $P_0$ is the orthogonal projection onto span$\{E_0\}$, while $P_{\ge0}$ is the orthogonal projection onto span$\{E_k, k\ge0\}$.
\end{prop}
\begin{proof}
This identification immediately follows from the the definition of $\mathcal{A}$ and the following decomposition of diagonal elements of $\mathcal{A}$:
\begin{equation*}
a(\K)=\sum_{k\leq -k_0}a(-\infty)(I-P_{\ge0})+\sum_{-k_0<k< k_0}a(k)U^nP_0U^{-n}+\sum_{k\ge k_0}a(\infty)P_{\ge0},
\end{equation*}
valid for $a(k)\in c_{00}^\pm(\Z)$. 
\end{proof}

Rotational symmetry on $A$ can be introduced in the following way. For each $\varphi\in[0,2\pi)$, let $\rho_\varphi : A\to A$ be an automorphism defined by $$\rho_\varphi(a(\K)) = a(\K)\textrm{ and } \rho_\varphi(U) = e^{-i\varphi}U,$$ 
for a diagonal operator $a(\K)$.
It is well defined on all of $A$ because it preserves the relation \eqref{CommRel}. Alternatively, the action of $\rho_\varphi$ can be written down using the label operator $\K$ as:\begin{equation*}
\rho_\varphi(a)=e^{i\varphi\K}ae^{-i\varphi\K}.
\end{equation*}
It follows that  $\rho_\varphi : \mathcal{A}\to \mathcal{A}$.  Notice also that we have the identifications:
$$\mathcal{A}_{diag}:=\left\{a(\K) \ : \{a(k)\}\in c_{00}^\pm(\Z) \right\}=\{a\in\mathcal{A} : \rho_\varphi(a) = a\},$$  
and similarly
$$A_{diag}= \left\{a(\K) \ : \{a(k)\}\in c(\Z) \right\}=\{a\in A : \rho_\varphi(a) = a\}.$$

The algebra $A$ also has a reflection symmetry that can be defined on the generators via 
\begin{equation}\label{little_theta}
\theta(U) = U \textrm{ and } \theta(a(\K)) = a(-\K).
\end{equation}
We will verify below that  $\theta$ extends to an anti-homomorphism $\theta : A\to A$, thus equipping the whole quantum cylinder with a natural reflection. Additionally, $\theta$ preserves the dense subalgebra $ \mathcal{A}$.

A natural analog of the classical space of square-integrable functions on the cylinder is the following GNS Hilbert space $\mathcal{H}$ for $A$ with respect to the weight $a\mapsto \textrm{tr}(a)$, which in turn is the analog of the Lebesque measure on the classical cylinder. It is the completion of the algebra $A$ with respect to the norm $\|a\|^2 =\textrm{ tr}(a^*a)$ for $a\in A$. Explicitly, we can view $\mathcal{H}$ in the following way:
\begin{equation}\label{Hdef_ref}
\mathcal{H} = \left\{f = \sum_{n\in\Z}U^nf_n(\K) : \sum_{n,k\in\Z}|f_n(k)|^2< \infty\right\}.
\end{equation}

The Hilbert space $\mathcal{H}$ carries natural representations of $A$. Let $\pi: A\to B(\mathcal{H})$ be the representation of $A$ in the bounded operators of $\mathcal{H}$ given by left multiplication:
$$ \pi(a)f = af,$$ 
where $a\in A$ and $f\in\mathcal{H}$.  We also need $\pi':A\to B(\mathcal{H})$ given by: 
$$\pi'(a)f = fa.$$  
Then the map $\pi$ is a continuous $*$ - homomorphism and $\pi'$ is a continuous $*$ - preserving anti-homomorphism. Both maps have trivial kernel.

Let $\mathcal{D}$ be the following subspace of $\mathcal{H}$:
\begin{equation}\label{Ddef_ref}
\mathcal{D} = \left\{f = \sum U^nf_n(\K) : f_n(k)\in c_{00}(\Z), \textrm{ finite sum}\right\}.
\end{equation}
Notice that $\mathcal{D}$ is dense in $\mathcal{H}$ and a straightforward calculation shows that the representations $\pi$ and $\pi'$ preserve $\mathcal{D}$. The subspace is a natural domain for unbounded operators considered in later sections.

The symmetries of $A$ can be implemented in $\mathcal{H}$ as follows. For the rotations, if $\varphi\in[0,2\pi)$, we define $U_\varphi :\mathcal{H}\to\mathcal{H}$ by the formula:\begin{equation*}
U_\varphi f = \sum_{n\in\Z} U^n e^{in\varphi} f_n(\K).
\end{equation*}
It follows from the definitions that we have $U_\varphi(\mathcal{D})=\mathcal{D}$ and
$$U_\varphi\pi(a)U_\varphi^{-1}=\pi(\rho_\varphi(a)).
$$

The implementation of reflection $\theta$ on $\mathcal{H}$ is given by the following operator $\Theta: \mathcal{H}\to \mathcal{H}$ given by the formula:
\begin{equation}\label{theta_decomp}
\Theta(f) = \Theta\left(\sum_{n\in\Z}U^nf_n(\K)\right) = \sum_{n\in\Z}f_n(-\K)U^n = \sum_{n\in\Z}U^nf_n(-\K-n).
\end{equation}
The relevant properties of $\Theta$ are summarized in the following propositions.
\begin{prop}
The operator $\Theta:\mathcal{H}\to\mathcal{H}$ satisfies $\Theta^2 = I$ and is a self-adjoint operator.
\end{prop}
\begin{proof}
The statements follow from direct calculations. If $f\in\mathcal{H}$, we have:
\begin{equation*}
\Theta^2(f) = \Theta\left(\sum_{n\in\Z}U^nf_n(-\K-n)\right) = \sum_{n\in\Z}U^n(-(-\K-n)-n) = \sum_{n\in\Z}U^nf_n(\K).
\end{equation*}
Thus $\Theta^2 = I$.  Similarly, if $f,g\in\mathcal{H}$ we have:
\begin{equation*}
\langle \Theta(f), g\rangle = \textrm{tr}\left((\theta f)^*g\right) = \textrm{tr}\left(\sum_{n\in\Z}\overline{f}_n(-\K-n)U^{-n}\sum_{m\in\Z}U^mg_m(\K)\right) = \sum_{n,k\in\Z}\overline{f}_n(-k-n)g_n(k),
\end{equation*}
where we used the fact that the trace is nonzero if and only if $n=m$.  Next, changing variables by sending $k\mapsto -k-n$ and resumming, we obtain:
\begin{equation*}
\langle \Theta(f), g\rangle  = \sum_{n,k\in\Z}\overline{f}_n(k)g_n(-k-n) = \textrm{tr}\left(\sum_{n\in\Z}\overline{f}_n(\K)U^{-n}\sum_{m\in\Z}U^mg_m(-\K-m)\right) = \langle f, \Theta(g)\rangle,
\end{equation*} 
which completes the proof.
\end{proof}

The interplay between representations $\pi$, $\pi'$ and $\Theta$ lets us extend the map $\theta$ defined on generators in \eqref{little_theta} to an anti-homomorphism on all of $A$.
\begin{prop}\label{morph_anti_on_gens}
We have $\Theta\pi(U)\Theta = \pi'(U)$ and $\Theta\pi(a(\K))\Theta = \pi'(a(-\K))$.   Consequently, there is an anti-homomorphism $\theta: A\to A$ such that 
$$\Theta\pi(a)\Theta = \pi'(\theta(a))$$ 
for all $a\in A$.
\end{prop}
\begin{proof}
The first part of the proposition follows from calculations. For $f\in\mathcal{H}$ we have:
\begin{equation*}
\begin{aligned}
\Theta\pi(U)\Theta &= \Theta\left(\sum_{n\in\Z}U^{n+1}f_n(-\K-n)\right) %= \Theta\left(\sum_{n\in\Z}U^nf_{n-1}(-\K-n+1)\right) \\
= \sum_{n\in\Z}U^{n+1}f_{n}(-(-\K-n+1)-n) \\
&= \sum_{n\in\Z}U^{n+1}f_n(\K+1) 
= \sum_{n\in\Z}U^nf_n(\K)U = \pi'(U)f.
\end{aligned}
\end{equation*}
Similarly,  we have 
\begin{equation*}
\begin{aligned}
\Theta\pi(a(\K))\Theta f &= \Theta\left(\sum_{n\in\Z}U^na(\K+n)f_n(-\K-n)\right) = \sum_{n\in\Z}U^na(-\K-n+n)f_n(\K) \\
&=\sum_{n\in\Z}U^nf_n(\K)a(-\K) = \pi'(a(-\K))f.
\end{aligned}
\end{equation*}
Since the range of $\pi'$ is closed, continuity establishes the existence of  $\theta$ such that $\Theta\pi(a)\Theta = \pi'(\theta(a))$. Uniqueness follows because $\pi'$ is faithful. This completes the proof.
\end{proof}

The last task in this section is to identify appropriate subspaces of $\mathcal{H}$ to be the space of noncommutative $L^2$ functions supported on the right and left half-cylinder.  Identifying $\mathcal{H}\cong \ell^2(\Z^2)$, we see that the operator $\Theta$ comes from the following natural space map
\begin{equation*}
\Z^2\to\Z^2: (n,k)\mapsto(n,-k-n).
\end{equation*}
A simple calculation verifies that the fixed points of the above transformation are those $(n,k)$ such that $n+2k = 0$.  Thus we obtain two natural subspaces of $\mathcal{H}:$
\begin{equation}\label{pos_neg_subspaces}
\begin{aligned}
&\mathcal{H}^+ = \{f\in\mathcal{H} : f_n(k) = 0,\textrm{ for }n+2k<0\} \\
&\mathcal{H}^- = \{f\in\mathcal{H} : f_n(k) = 0,\textrm{ for }n+2k>0\}. 
\end{aligned}
\end{equation}
We have the following key property.
\begin{prop}\label{theta_H_plus_H_minus}
With the above notation, we have $\Theta:\mathcal{H}^+\to\mathcal{H}^-$.
\end{prop}
\begin{proof}
For $f\in\mathcal{H}^+$ if $g=\Theta f$, then we have $g_n(k) = f_n(-k-n)$ and moreover $g_n(k)=0$, provided $n+2(-k-n)<0$.  Thus $g_n(k)=0$ if $n+2k>0$, completing the proof.
\end{proof}

\section{Classification of Derivations in the Quantum cylinder}

The purpose of this section is to give a description of unbounded derivations in $A$ defined on $\mathcal{A}$. One can think of derivations as noncommutative analogs of vector fields, and they will be used in the next section as building blocks in constructing quantum Laplace-like operators.

Analogous classification of derivations  in the quantum disk, and in several other Toeplitz-type algebras, was discussed in \cite{KMRSW1} and \cite{KMRSW2}, and similar decompositions were previously introduced in \cite{BEJ}, \cite{H}, and \cite{J}, see also references therein.   

The key to understanding the structure of derivations in $A$ is the following simple observation.

\begin{prop}
Let $d:\mathcal{A}\to A$ be any derivation, then $d: \mathcal{A}\cap\mathcal{K}\to \mathcal{K}$.  Consequently $d$ defines a  derivation on equivalence classes: 
\begin{equation*}
[d]: [\mathcal{A}]\to A/\mathcal{K} \cong C(S^1)\oplus C(S^1).
\end{equation*}
\end{prop}\begin{proof}
Notice that the projection $P_0$ is in $\mathcal{A}\cap\mathcal{K}$. Applying the Leibniz rule to $P_0^2=P_0$ we get:
$$d(P_0)=P_0d(P_0)+d(P_0)P_0.
$$
It follows that $d(P_0)$ is in $\mathcal{K}$. But any element of $\mathcal{A}\cap\mathcal{K}$ is a finite sum of operators of the form $U^nP_0U^m$, and the Leibniz rule again implies that the action of $d$ on such elements results in a compact operator.
\end{proof}

We want to understand the properties of the correspondence $d\mapsto [d]$. To describe its kernel we need the following concepts.

Any derivation $d:\mathcal{A}\to A$ that satisfies the relation
\begin{equation*}
d(\rho_\varphi(a)) = e^{-in\varphi}\rho_\varphi(d(a))
\end{equation*}
will be referred to as a $n$-{\it covariant} derivation.
We say that a function $\beta:\Z \to\C$ has {\it convergent increments}, if the sequence of differences $\{\beta(k)-\beta(k-1)\}$ is convergent i.e. 
$$\{\beta(k)-\beta(k-1)\}\in c(\Z).$$  
The set of all such functions will be denoted by $c_{inc}(\Z)$. 
The following proposition classifies all $n$-covariant derivations $d:\mathcal{A}\to A$.  

\begin{prop}\label{n_covar_der_rep}
If $d$ is a $n$-covariant derivation $d:\mathcal{A}\to A$, then there exists a function $\beta\in c_{inc}(\Z)$, which is unique in $n\ne 0$ and unique modulo an additive constant when $n=0$, such that 
$$d(a) = [U^n\beta(\K), a]$$ 
for $a\in\mathcal{A}$.
\end{prop}
\begin{proof}
The proof is a slight extension of Proposition 3.2 in \cite{KMR}, see also Theorem 3.4 of \cite{KMRSW2}.
Defining a linear map $\tilde{d}:\mathcal{A}\to A$ via the formula:
$$\tilde{d}(a) := U^{-n}d(a),$$
we see from the covariance property of $d$ that $\tilde{d}:\mathcal{A}_{diag}\to A_{diag}$. Additionally, $\tilde{d}$ satisfies a twisted Leibniz rule of the form:
\begin{equation*}
\tilde{d}(a(\K)b(\K)) = \tilde{d}(a(\K))b(\K) + a(\K+n)\tilde{d}(b(\K)).
\end{equation*}
Because $\mathcal{A}_{diag}$ is commutative we have: $\tilde d(a(\mathbb K)b(\mathbb K))= \tilde d(b(\mathbb K)a(\mathbb K))$, so the following relation follows: 
\begin{equation*}
\tilde d(a(\mathbb K))[b(\mathbb K )-b(\mathbb K +n)]= \tilde d(b(\mathbb K))[a(\mathbb K )-a(\mathbb K +n)]. 
\end{equation*}
We can now easily see, as in \cite{KMRSW2}, that there is a diagonal operator $\beta(\K)$ such that: 
\begin{equation*}
\tilde{d}(a(\K)) = \beta(\K)(a(\K) - a(\K+n)),
\end{equation*}
and the result follows.
\end{proof}

Recall that $d$ is called {\it approximately inner} if there are $a_n\in A$  such that 
$$d(a) = \lim_{n\to\infty}[a_n,a]$$ 
for $a\in\mathcal{A}$. \begin{prop}\label{approx_inner_coef}
If $d$ is a $n$-covariant derivation $d:\mathcal{A}\to A$, then $d$ is approximately inner if and only if $\{\beta(k+1)-\beta(k)\}\in c_0(\Z)$.
\end{prop}\begin{proof}
Full details of an analogous result for a class of more complicated examples is given in Theorem 3.10 of \cite{KMRSW2}, therefore we only give a brief outline here. 

If $\{\beta(k+1)-\beta(k)\}\in c_0(\Z)$ then the corresponding derivation $d(a) = [U^n\beta(\K), a]$ can be approximated by inner derivations $d_M(a) = [U^n\beta_M(\K), a]$, where $\beta_M(k)$ are eventually constant and such that $\beta_M(k+1)-\beta_M(k)=\beta(k+1)-\beta(k)$ for $|k|\leq M$ while $\beta_M(k+1)-\beta_M(k)=0$ if $|k|>M$.

On the other hand if a $n$-covariant derivation $d(a) = [U^n\beta(\K), a]$ is approximately inner then we can arrange that it can be approximated by inner $n$-covariant derivations of the form  $d_M(a) = [U^n\beta_M(\K), a]$ with $\beta_M(k)\in c(\Z)$. Since $\{\beta_M(k+1)-\beta_M(k)\}\in c_0(\Z)$ we also get $\{\beta(k+1)-\beta(k)\}\in c_0(\Z)$.
\end{proof}

With this preparation we are ready to describe the kernel of the quotient map $d\mapsto [d]$.
\begin{theo}
Let $d:\mathcal{A} \to A$ be any derivation.  Then $[d] = 0$  if and only if $d$ is approximately inner.
\end{theo}\begin{proof}
Let $d$ be approximately inner, then there exists $a_n\in A$ such that 
\begin{equation*}
d(a) = \lim_{n\to\infty}[a_n,a]
\end{equation*}
for all $a\in\mathcal{A}$.  It was shown in \cite{KMRS} that $\mathcal{K}$ is the commutator ideal.  Thus $[a_n,a]\in\mathcal{K}$ for all $n$ and $a\in\mathcal{A}$.  Since the norm limit of compact operators is compact, we have $d(a)\in\mathcal{K}$ for all $a\in\mathcal{A}$ and hence $[d]=0$.  

On the other hand, suppose $[d]=0$.  Define $d_n:\mathcal{A}\to A$ via
\begin{equation*}
d_n(a) = \frac{1}{2\pi}\int_0^{2\pi}e^{in\varphi}\rho_\varphi^{-1}(d(a))\ d\varphi .
\end{equation*}
It can be readily checked that $d_n$ is the $n$-covariant component of $d$, thus by Proposition \ref{n_covar_der_rep}, there exists a sequence $\{\beta_n(k)\}$ such that
\begin{equation*}
d_n(a) = [U^n\beta_n(\K),a].
\end{equation*}
If $[d]=0$ then it easily follows that $[d_n]=0$.  Since the operator
$$d_n(U) = U^{n+1}(\beta_n(\K+1)-\beta_n(\K))$$ 
is compact, it therefore follows that 
$$(\beta_n(k+1)-\beta_n(k))\to 0$$
as $k\to\pm\infty$. Then Proposition \ref{approx_inner_coef} implies that $d_n$ is approximately inner.  Notice that by the usual C\'esaro approximation argument, see \cite{KMRSW1}, \cite{KMRSW2}, we have:
\begin{equation*}
d(a) = \lim_{m\to\infty} \tilde{d_m}(a),
\end{equation*}
where
\begin{equation*}
\tilde{d_m}(a) = \frac{1}{m+1}\sum_{n=0}^m\sum_{j=-n}^nd_j(a).
\end{equation*}
Since all $d_j$'s are approximately inner it follows that $\tilde{d_m}$ is approximately inner. Consequently, there exists $x_{m,j}\in A$ such that
for all $a\in\mathcal{A}$ we have:
\begin{equation*}
\tilde{d_m}(a) = \lim_{j\to\infty} [x_{m,j},a].
\end{equation*}

Recall the identification $\mathcal{A} = Pol(U_r,U_r^*)$ and so to check that $d$ is approximately inner, we need only to verify convergence of inner approximations on $U_r$ and $U_r^*$.   Given $m$, choose $j(m)$ such that:
\begin{equation*}
\|\tilde{d_m}(U_r) - [x_{m,j(m)}, U_r]\| < \frac{1}{m},
\end{equation*}
and also such that the same condition holds for $U_r^*$.
Then we claim that we have convergence:
\begin{equation*}
d(a) = \lim_{m\to\infty} [x_{m,j(m)},a],
\end{equation*}
for all $a\in\mathcal{A}$.  Notice that by the triangle inequality we can estimate:
\begin{equation*}
\|d(U_r) -[x_{m,j(m)}, U_r]\| \le \|d(U_r)-\tilde{d_m}(U_r)\| + \|\tilde{d_m}(U_r) - [x_{m,j(m)}, U_r]\| <\frac{\varepsilon}{2} + \frac{1}{m} < \varepsilon,
\end{equation*}
as $\tilde{d_m}(U_r)\to d(U_r)$ as $m\to\infty$.  Since the claim holds on $U_r$, $U_r^*$, $d$ is linear and satisfies the Leibniz rule, the claim holds for all $a\in\mathcal{A}$, thus proving $d$ is approximately inner.
\end{proof}

It should be noted that this type of argument only works since $\mathcal A$ is finitely generated. 
We also get the following immediate corollary.

\begin{cor}
If $d_1,d_2:\mathcal{A}\to A$ are any two derivations such that $[d_1] = [d_2]$, then $d_1-d_2$ is approximately inner.
\end{cor}

Next we want to describe the range of the map $d\mapsto [d]$. For this we notice that we can identify $[\mathcal{A}] = Pol(S^1)\oplus Pol(S^1)$ in $A/\mathcal{K} = C(S^1)\oplus C(S^1)$ and we can easily classify all derivations $\delta: Pol(S^1)\oplus Pol(S^1\to C(S^1)\oplus C(S^1)$.
\begin{prop}\label{class_der}
If $\delta: [\mathcal{A}]\to C(S^1)\oplus C(S^1)$ is any derivation in  $C(S^1)\oplus C(S^1)$ then there exist functions $f,g\in C(S^1)$ such that
\begin{equation}\label{dfgdef_ref}
\delta = \delta_{f,g} := \left(f(x)\frac{1}{i}\frac{d}{dx}, g(x)\frac{1}{i}\frac{d}{dx}\right).
\end{equation}
\end{prop}\begin{proof}
The result will follow if we show that if $\delta : Pol(S^1)\to C(S^1)$ is any derivation then there exists a function $f(x)\in C(S^1)$ such that 
\begin{equation*}
\delta = f(x)\frac{1}{i}\frac{d}{dx}.
\end{equation*}
Since $Pol(S^1)$ is the algebra of trigonometric polynomials, a derivation is completely determined by its action on the generator $e^{ix}$.  Consequently, there exist a continuous function $f$  such that
\begin{equation*}
\delta(e^{ix}) =f(x)e^{ix},
\end{equation*} 
and it follows that $\delta$ is given by the formula above.
\end{proof}

It turns out that the range of the map $d\mapsto [d]$ is the space of all derivations, i.e. any derivation $\delta: [\mathcal{A}]\to A/\mathcal{K}$ can be lifted to a derivation $d: \mathcal{A}\to A$. This kind of lifting problem is often the most difficult step in classification of derivations with respect to an ideal in the algebra.
\begin{theo}
If $\delta_{f,g}: [\mathcal{A}]\to C(S^1)\oplus C(S^1)$ is the derivation given by \eqref{dfgdef_ref}, then there exists a derivation $d_{f,g}:\mathcal{A}\to A$ such that $[d_{f,g}] = \delta_{f,g}$.
\end{theo}\begin{proof}
Given any two continuous functions $f$ and $g$, define a Toeplitz-like operator $T_{f,g}: C(S^1)\oplus C(S^1)\to A$ by the following formula:
\begin{equation*}
T_{f,g} = P_{\ge0}f(U)P_{\ge0} + P_{<0}g(U)P_{<0},
\end{equation*}
see Proposition \ref{useful_gens}. By the continuous functional calculus, $T_{f,g}$ is well-defined and bounded, moreover $T_{f,g}\in A$.  

First we prove the lifting result when $f,g$ are trigonometric polynomials and then extend it to general continuous functions.  For $f,g\in Pol(S^1)$ define the following operators:
\begin{equation*}
T_{f,g}^+ = \sum_n U^n T_n^+(\K) \quad\textrm{ and }\quad T_{f,g}^- = \sum_n T_n^-(\K)U^n,
\end{equation*}
where the sums over $n$ are finite and
\begin{equation*}
T_n^+(k) = \left\{
\begin{aligned}
&f_n && n,k>0 \\
&g_n && n,k\le 0\\
&0 &&\textrm{else}
\end{aligned}\right. \quad\textrm{ and }\quad T_n^-(k) = \left\{
\begin{aligned}
&f_n && k\ge0, n\le 0\\
&g_n && k<0, n>0 \\
&0 &&\textrm{else,}
\end{aligned}\right.
\end{equation*}
where $f_n$ and $g_n$ are the Fourier coefficients of $f$ and $g$ respectively.  We define a derivation $d_{f,g}$ on $\mathcal{A}$ by
\begin{equation}\label{der_on_trig_poly}
d_{f,g}(a) = [T_{f,g}^+\K + \K T_{f,g}^-,a].
\end{equation}
Since $d_{f,g}$ is defined through a commutator it's clear that $d_{f,g}$ satisfies the Leibniz rule for derivations, thus the only items that must be checked is that $d_{f,g}(a)\in A$ for $a\in\mathcal{A}$ and $[d_{f,g}] = \delta_{f,g}$.  Since $\mathcal{A}$ is generated by $U$, $P_{\ge0}$, and $P_0$, we only need to check the conditions on those three generators.  For $U$ we have the following:
\begin{equation*}
\begin{aligned}
&d_{f,g}(U) = \\
&=\sum_n U^{n+1}\left((\K+1)T_n^+(\K+1) - \K T_n^+(\K)\right) + \sum_n \left(\K T_n^-(\K) - (\K-1)T_n^-(\K-1)\right)U^{n+1} \\
%&=\sum_nU^{n+1} T_n^+(\K+1) + \sum_n T_n^-(\K-1) U^{n+1}\\= U\left(\sum_n U^n T_n(\K+1) + \sum_n T_n^-(\K) U^n\right) \\
&= UP_{\ge0}\left(\sum_{n>0} U^nf_n P_{\ge0} + \sum_{n\le0}f_nU^nP_{\ge0}\right) + UP_{<0}\left(\sum_{n\le0}U^ng_nP_{<0} + \sum_{n>0}g_nU^nP_{<0}\right) \\
&= UT_{f,g},
\end{aligned}
\end{equation*}
where we used the definition of $T_n^\pm(k)$ and the definition of the orthogonal projections.  In particular, we have $d_{f,g}(U)\in A$.  

A similar computation for $a(\K)\in \mathcal{A}_{diag}$ gives the following general formula:
\begin{equation*}
\begin{aligned}
d_{f,g}(a(\K)) &= \left[\left(\sum_{n>0}U^nf_n\right)P_{>0}\K,a(\K)\right] + \left[\left(\sum_{n\le0}U^ng_n\right)P_{\le0}\K,a(\K)\right] \\
&\qquad + \left[\K P_{<0}\left(\sum_{n>0}g_nU^n\right),a(\K)\right] + \left[\K P_{\ge0}\left(\sum_{n\le0}f_nU^n\right),a(\K)\right].
\end{aligned}
\end{equation*}
It is clear from the above formula that we have 
$$d_{f,g}(P_{\ge0}) = d_{f,g}(P_0) = 0.$$  
Thus $d_{f,g}$ is  a well-defined derivation $\mathcal{A}\to A$.  

Finally we need to check that $[d_{f,g}] = \delta_{f,g}$. The ``restriction'' to the boundary map 
$$\sigma=(\sigma_+, \sigma_-) : A \to C(S^1)\oplus C(S^1)$$ 
is defined by $\sigma_\pm : A\to C(S^1)$ with 
$$\sigma_\pm(U) = e^{ix} \textrm{ and } \sigma_\pm(a(\K)) = a(\pm\infty),$$ 
where $a(\pm\infty)$ are the limits of $a(k)$ at $\pm\infty$.  The map $\sigma$ is the factor homomorphism and we have 
$$[d_{f,g}(a)] = \sigma(d_{f,g}(a))$$ 
for all $a\in\mathcal{A}$. From the definition of $T_{f,g}$ we have the formula: 
$$\sigma(T_{f,g}) = (f(x),g(x)),$$ 
and, obviously, also 
$$\sigma(d_{f,g}(P_{\ge0})) = \sigma(d_{f,g}(P_0)) = 0.$$  
Consequently, we obtain:
\begin{equation*}
\sigma(d_{f,g}(U)) = (\sigma_+ \oplus \sigma_-)(UT_{f,g}) = (f(x)e^{ix}, g(x)e^{ix}) = \delta_{f,g}(e^{ix}).
\end{equation*}
Therefore, for $f,g\in Pol(S^1)$, there exists a derivation $d_{f,g}:\mathcal{A}\to A$ such that $[d_{f,g}] = \delta_{f,g}$.

Next we consider arbitrary $f,g\in C(S^1)$.  We will define $d_{f,g}: \mathcal{A}\to A$ by an approximation argument.
There exist sequences $\{f_m\},\{g_m\}\in Pol(S^1)$ such that $f_m,g_m$ converge to $f,g$ in $C(S^1)$ respectively.  For any $a\in\mathcal{A}$ we want to define $d_{f,g}(a)$ as the limit:
\begin{equation*}
d_{f,g}(a) = \lim_{m\to\infty} d_{f_m,g_m}(a),
\end{equation*}
where $d_{f_m,g_m}(a)$ were defined in equation \eqref{der_on_trig_poly}.  Such a limit exists because it exists on the generators $U$, $P_{\ge0}$, and $P_0$ of $\mathcal{A}$. In fact, we have:
$$d_{f,g}(U) = \lim_{m\to\infty} UT_{f_m,g_m}= UT_{f,g},$$
and similarly
$$d_{f,g}(P_{\ge0}) = d_{f,g}(P_0) = 0.$$  
This completes the proof. Notice that we could not have simply defined $d_{f,g}$ by the above formulas on generators as it is not immediately clear that they unambiguously extend to a derivation on all of $\mathcal A$.
\end{proof}

The following observation immediately follows from the above proof.
\begin{cor} For any $f,g\in C(S^1)$ there exists unique derivation $d_{f,g}:\mathcal{A}\to A$ such that:
$$d_{f,g}(U) =  UT_{f,g},\ \ d_{f,g}(P_{\ge0}) = d_{f,g}(P_0) = 0.$$
Moreover, $[d_{f,g}] = \delta_{f,g}$ where $\delta_{f,g} :[\mathcal{A}] \to C(S^1) \oplus C(S^1)$ is given by
\begin{equation*}
\delta_{f,g} = \left(f(x)\frac{1}{i}\frac{d}{dx}, g(x)\frac{1}{i}\frac{d}{dx}\right).
\end{equation*}

\end{cor}
Putting together all of these results, we get the following classification theorem, which is the first main result of this paper.

\begin{theo}
Let $d:\mathcal{A}\to A$ be any derivation.  There are unique functions $f,g\in C(S^1)$ such that $d = \tilde{d} + d_{f,g}$, where $\tilde{d}$ is an approximately inner  derivation.
\end{theo}

\section{Reflection Positivity of Laplace-type Operators}

This section contains our main results on reflection positivity of Laplace-type operators constructed using special kind of derivations described and classified above.

Any derivation $d:\mathcal{A}\to A$ that satisfies the relation: 
$$\rho_\varphi(d(a)) = d(\rho_{\varphi}(a))$$ 
will be referred to as a $\rho_\varphi$-{\it invariant} derivation. Similarly, any derivation $d:\mathcal{A}\to A$ that satisfies the relation  
$$d(\rho_\varphi(a))= e^{-i\varphi}\rho_\varphi(d(a))$$ 
for all $a\in \mathcal{A}$ will be referred to as a $\rho_\varphi$-{\it covariant} derivation. Those are special cases of $n$-covariant derivations from the previous section, corresponding to $n=0$ and $n=1$ respectively. 

In order to construct interesting, geometrical Laplace-type operators we want to implement such derivations as operators in the Hilbert space $\mathcal{H}$ defined in \eqref{Hdef_ref}.
In \cite{KMRSW1} and \cite{KMR}, the invariant and covariant derivations were easily implementable because the representation of the algebra in the Hilbert space considered in those papers had a cyclic vector.  In the present situation this is no longer the case, thus we need a new argument to show that the invariant and covariant derivations can be implemented us unbounded operators on domain $\mathcal{D}$ defined in \eqref{Ddef_ref}.

The following proposition gives a description of such invariant and covariant implementations and uses an approximate identity argument. 

\begin{prop}\label{der_implement} (Implementation)

\begin{itemize} 
\item Let $d:\mathcal{A}\to A$, $d(a)=[\beta(\K),a]$, be an invariant derivation in $A$. Then 
there exists a densely-defined operator $D :\mathcal{D}\to \mathcal{H}$ such that 
\begin{equation*}
[D,\pi(a)]f = \pi(d(a))f,
\end{equation*}
for $f\in\mathcal{D}$, $a\in\mathcal{A}$ and 
$$U_\varphi DU_\varphi^{-1} f = Df$$ 
for $f\in\mathcal{D}$.  Moreover there exists a sequence $\alpha(k)$ such that the implementation has the following form:
\begin{equation*}
Df = \beta(\K)f - f\alpha(\K).
\end{equation*}

\item Let $d:\mathcal{A}\to A$, $d(a)=[U\beta(\K),a]$, be a covariant derivation in $A$. Then 
there exists a densely-defined operator $D :\mathcal{D}\to \mathcal{H}$ such that 
\begin{equation*}
[D,\pi(a)]f = \pi(d(a))f,
\end{equation*}
for $f\in\mathcal{D}$, $a\in\mathcal{A}$ and 
$$U_\varphi DU_\varphi^{-1} f = e^{i\varphi} Df$$ for 
$f\in\mathcal{D}$.  Moreover there exists a sequence $\alpha(k)$ such that the implementation has the following form:
\begin{equation*}
Df = U\beta(\K)f - fU\alpha(\K).
\end{equation*}
\end{itemize}
\end{prop}\begin{proof} Both parts have almost identical proofs; we will concentrate below on the covariant case.
Observe that we have a natural decomposition: 
\begin{equation}\label{Hdecomp}
\mathcal{H} = \bigoplus_{n\in\Z}\mathcal{H}_n,
\end{equation}
where 
$$\mathcal{H}_n:= \left\{f = U^nf_n(\K): \sum_{k\in\Z}|f_n(k)|^2 <\infty\right\}.
$$
This is precisely the spectral decomposition for $U_\varphi$ as we have $U_\varphi: \mathcal{H}_n\to\mathcal{H}_n$ and 
$$U_\varphi|_{\mathcal{H}_n} = e^{in\varphi}I.$$ 
We have analogous subspace decomposition: $\mathcal{D} = \bigoplus_{n\in\Z} \mathcal{D}_n$ with $\mathcal{D}_n\subseteq \mathcal{H}_n$.  

Suppose $D$ is an implementation of a covariant derivation as in the statement of the proposition. Given $f\in\mathcal{H}_n$, we compute:
\begin{equation*}
U_\varphi(Df) =e^{i\varphi} D(U_\lambda f) = e^{i(n+1)\varphi}Df,
\end{equation*}
implying that $D:\mathcal{D}_n\to \mathcal{H}_{n+1}$.  

Consider the following characteristic function:
\begin{equation*}
\chi_N(k) = \left\{
\begin{aligned}
&1 &&\textrm{for }|k|\le N \\
&0 &&\textrm{otherwise.}
\end{aligned}\right.
\end{equation*}
Then $\chi_N(\K)\in\mathcal{D}$ and so there exists $\gamma_N(\K)\in\mathcal{H}_N$ such that $$D(\chi_N(\K)) = U\gamma_N(\K).$$  
Since for $N\le M$ we have 
$$\chi_N(\K)\chi_M(\K) = \chi_N(\K),$$ 
applying $D$ to both sides we obtain: 
$$D(\pi(\chi_N(\K))\chi_M(\K)) = D(\chi_N(\K)),$$ and therefore
\begin{equation*}
[D,\pi(\chi_N(\K))]\chi_M(\K) + \chi_N(\K)D(\chi_M(\K)) = D(\chi_N(\K)).
\end{equation*}
Using the properties of $D$ we obtain the following relation:
\begin{equation*}
\pi(d(\chi_N(\K)))\chi_M(\K) + \chi_N(\K)U\gamma_M(\K) = U\gamma_N(\K),
\end{equation*}
or equivalently:
\begin{equation*}
U\beta(\K)(\chi_N(\K)-\chi_N(\K+1)) + U\chi_N(\K +1)\gamma_M(\K) = U\gamma_N(\K).
\end{equation*}
It follows from the above equation, and the definition of $\chi_N(\K)$, that if $-N\le k, k+1\le N$, we have: 
$$\gamma_M(k) = \gamma_N(k)$$
for every $M>N$.  Thus, if we fix $k$, then $N\mapsto \gamma_N(k)$ is eventually constant and hence the following limit exists:
\begin{equation*}
\gamma(k) = \lim_{N\to\infty}\gamma_N(k).
\end{equation*}

Given $f\in\mathcal{D}$ choose $N$ large enough such that $f\chi_N(\K) = f$ and $\gamma_N(\K) = \gamma(\K)$.  Such a choice can be made since $f$ has finite support. Then, we obtain the following formula:
\begin{equation*}
\begin{aligned}
Df &= D(f\chi_n(\K)) = [D,\pi(f)]\chi_n(\K) + \pi(f)D(\chi_N(\K))= \pi(d(f))\chi_N(\K) + fU\gamma_N(\K) \\
&= \pi(d(f)) + fU\gamma(\K)=U\beta(\K)f - fU(\beta(\K)-\gamma(\K)),
\end{aligned}
\end{equation*}
as in the statement of the proposition.

Conversely, for any sequence  $\alpha(k)$ the formula $Df = U\beta(\K)f - fU\alpha(\K)$ clearly defines an operator $D :\mathcal{D}\to \mathcal{H}$ implementing the derivation $d(a)=[U\beta(\K),a]$, completing the proof.
\end{proof}

\subsection{Reflection Positivity: Invariant Case}
We describe here the easier of two interesting reflection positivity results.
Let $d_\beta : \mathcal{A}\to A$, given by $d_\beta(a) = [\beta(\K),a]$, be an invariant derivation.  We proved in Proposition \ref{der_implement} that such a derivation can be implemented by a densely-defined operator $D_{\alpha,\beta} : \mathcal{D} \to \mathcal{H}$ given by:
\begin{equation*}
D_{\alpha,\beta}f = \pi(\beta(\K))f - \pi'(\alpha(\K))f = \sum_{n\in\Z}U^n(\beta(\K+n) - \alpha(\K))f_n(\K).
\end{equation*}

If $\alpha(k)$ and $\beta(k)$ are real then the operator $D_{\alpha,\beta}$ is symmetric and, instead of constructing a Laplace-type operator, we can work with $D_{\alpha,\beta}$ directly.

If $\alpha(k) \neq \beta(j)$ for all $j,k\in\Z$, then it is clear that $D_{\alpha,\beta}$ is an invertible operator and we have the following formula for its inverse:
\begin{equation*}
D_{\alpha,\beta}^{-1}f = \sum_{n\in\Z}U^n(\beta(\K+n) - \alpha(\K))^{-1}f_n(\K).
\end{equation*}

Part of the reflection positivity scheme is that the implementation $D_{\alpha,\beta}$ has to be invariant with respect to $\Theta$:
$$\Theta D_{\alpha,\beta}\Theta = D_{\alpha,\beta}.$$  
When this happens is the subject of the following proposition.
\begin{prop}
With the above notation we have that $D_{\alpha, \beta}$ is a $\Theta$ - invariant implementation of an invariant derivation $d_\beta$ if and only if 
$$\alpha(k) = -\beta(-k)$$ for all $k\in\Z$.
\end{prop}\begin{proof}
Using Proposition \ref{morph_anti_on_gens}, we get:
\begin{equation*}
\Theta D_{\alpha,\beta}\Theta = \Theta\pi(\beta(\K))\Theta - \Theta\pi'(\alpha(\K))\Theta = \pi'(\beta(-\K))-\pi(\alpha(-\K)) = -D_{\theta\alpha,\theta\beta}
\end{equation*}
It follows that if $\Theta D_{\alpha,\beta}\Theta = D_{\alpha,\beta}$ we have to have for all $f\in\mathcal D$:
$$ (\beta(\K) + \alpha(-\K))f=f(\alpha(\K) + \beta(-\K)).
$$
Choosing $f=U^n\chi_N(\K)$ and varying $n$ and $N$ we conclude that we have to have the following relation for every $k$:
$$\alpha(k) + \beta(-k)= \textrm{const.}$$ 
This finishes the proof as, given $d_\beta$, the sequence $\beta(k)$ is determined up to additive constant.
\end{proof}

In light of this proposition, we will now label the $\Theta$ - invariant implementation by $D_\beta$, which is given by the formula 
\begin{equation*}
D_\beta = \pi(\beta(\K))+\pi'(\beta(-\K)),
\end{equation*}
or, equivalently,
\begin{equation*}
D_\beta f = \sum_{n\in\Z}U^n(\beta(\K+n) + \beta(-\K))f_n(\K)
\end{equation*}
for $f\in\mathcal{H}$.

The following is the reflection positivity inequality in the invariant case.
\begin{theo}
If $\beta(k)> 0$ for all $k\in\Z$, then $D_\beta$ is invertible and for all $f\in\mathcal{H}^+$ we have:
\begin{equation*}
\langle \Theta f, D_\beta^{-1}\rangle \ge 0.
\end{equation*}
\end{theo}\begin{proof}
The proof of the theorem is computational and it is simplified by the fact that the implementation operator $D_\beta$ is diagonal. Intriguingly, like in the continuous case, the left-hand side of the inequality can be expressed as an integral (sum) over ``the boundary", i.e. the reflection invariant subset $n+2k=0$ of $\Z^2$.

The inner product in the inequality can be separated into three parts as follows:
\begin{equation*}
\begin{aligned}
\langle \Theta f, D_\beta^{-1}\rangle &= \sum_{n,k\in\Z}\overline{f}_n(-k-n)(\beta(k+n)+\beta(-k))^{-1}f_n(k) \\
&= \left(\sum_{n+2k<0} + \sum_{n+2k=0} + \sum_{n+2k>0}\right)\overline{f}_n(-k-n)(\beta(k+n)+\beta(-k))^{-1}f_n(k).
\end{aligned}
\end{equation*}
Since $f\in\mathcal{H}^+$ this means $f_n(k) = 0$ on $n+2k<0$, thus the sum over $n+2k<0$ is zero.  Moreover by Proposition \ref{theta_H_plus_H_minus} we have $f_n(-k-n)=0$ on $n+2k>0$, hence the sum over $n+2k>0$ is zero.  Consequently we have
\begin{equation*}
\langle \Theta f, D_\beta^{-1}\rangle = \sum_{n+2k=0}\overline{f}_n(-k-n)(\beta(k+n)+\beta(-k))^{-1}f_n(k) = \frac{1}{2}\sum_{n-\textrm{even}}\left|f_n\left(-\frac{n}{2}\right)\right|^2\beta\left(\frac{n}{2}\right)^{-1}\ge 0. 
\end{equation*}

\end{proof}

\subsection{Reflection Positivity: Covariant Case}
As in the invariant case, let $d_\beta:\mathcal{A}\to A$ given by $d_\beta(a) = [U\beta(\K),a]$, be a covariant derivation in $A$. Its implementations, according to Proposition \ref{der_implement}, are operators $D_{\alpha,\beta}:\mathcal{D}\to \mathcal{H}$ given by:
\begin{equation*}
D_{\alpha,\beta}f = U\beta(\K)f - fU\alpha(\K) = \sum_{n\in\Z}U^{n+1}(\beta(\K+n)f_n(\K)-\alpha(\K)f_n(\K+1)),
\end{equation*}
or equivalently:
\begin{equation*}
D_{\alpha,\beta} = \pi(U\beta(\K)) - \pi'(U\alpha(\K)).
\end{equation*}
We now request $\Theta$ invariance of the ``Laplacian'' operator:
\begin{equation*}
\Theta D_{\alpha,\beta}^*D_{\alpha,\beta}\Theta = D_{\alpha,\beta}^*D_{\alpha,\beta}.
\end{equation*}
The conditions for this to be true are described in the next statement.
\begin{prop}
With the above notation we have $\Theta D_{\alpha,\beta}^*D_{\alpha,\beta}\Theta =  D_{\alpha,\beta}^*D_{\alpha,\beta}$ 
if and only if $\alpha(k) = \pm\beta(-k-1)$ for all $k\in\Z$.
\end{prop}
\begin{proof}
Using properties of $\pi$ and $\pi'$ we obtain the following formula for the adjoint of $D_{\alpha,\beta}$:
\begin{equation*}
D_{\alpha,\beta}^* = \pi(\overline{\beta}(\K)U^*) - \pi'(\overline{\alpha}(\K)U^*)\, .
\end{equation*}
A simple calculation yields a formula for the corresponding Laplacian:
\begin{equation*}
D_{\alpha,\beta}^*D_{\alpha,\beta} = \pi(|\beta(\K)|^2) -\pi'(\overline{\alpha}(\K)U^*)\pi(U\beta(\K)) - \pi(\overline{\beta}(\K)U^*)\pi'(U\alpha(\K)) + \pi'(|\alpha(\K-1)|^2)\, .
\end{equation*}
Using Proposition \ref{morph_anti_on_gens} we obtain:
\begin{equation*}
\Theta D_{\alpha,\beta}^*D_{\alpha,\beta}\Theta = \pi'(|\beta(-\K)|^2) - \pi(U^*\overline{\alpha}(-\K))\pi'(\beta(-\K)U) - \pi'(U^*\overline{\beta}(-\K))\pi(\alpha(-\K)U)\, .
\end{equation*}
When comparing terms of $\Theta D_{\alpha,\beta}^*D_{\alpha,\beta}\Theta$ and $D_{\alpha,\beta}^*D_{\alpha,\beta}$ we need the following observation.
\begin{lem} If all $a(k)$, $a'(k)$, $b(k)$ and $b'(k)$ are not identically zero, and we have:
\begin{equation*}
a(\K)fb(\K)=a'(\K)fb'(\K)
\end{equation*}
for all $f\in\mathcal D$, then there is a nonzero constant $\mu$ such that 
\begin{equation*}
a'(k)=\mu a(k) \textrm{ and } b'(k)=\frac 1\mu b(k)
\end{equation*}
for all $k\in\Z$.
\end{lem}
\begin{proof} Substituting $f=U^n\chi_N(\K)$ we conclude that we have to have:
\begin{equation*}
a(k)b(l)=a'(k)b'(l)
\end{equation*}
for all $k,l\in\Z$. Notice that, since all the sequences are not identically zero, the above equation implies that $a(k)=0$ if and only if $a'(k)=0$, and similarly, $b(l)=0$ if and only if $b'(l)=0$. Now choose $l_0$ such that $b(l_0)\ne 0$ and $b'(l_0)\ne 0$ and define $\mu$ to be:
\begin{equation*}
\mu=\frac{b(l_0)}{b'(l_0)},
\end{equation*}
and the result immediately follows.
\end{proof}

Applying the lemma we see that $\Theta D_{\alpha,\beta}^*D_{\alpha,\beta}\Theta = D_{\alpha,\beta}^*D_{\alpha,\beta}$
is true if and only if there is a constant $\mu\in\C$ such that with have the following equalities for all $k$:
\begin{equation*}
\begin{aligned}
|\beta(k)|
^2 &= |\alpha(-k-1)|^2\\
\overline{\beta}(k) &= \mu\overline{\alpha}(-k-1) \\
\alpha(k) &= \frac{1}{\mu}\beta(-k-1).
\end{aligned}
\end{equation*}
This is satisfied if and only if $|\mu|^2=1$ and $\overline{\mu}=\mu$.  Therefore, we have $\mu = \pm 1$, which completes the proof.
\end{proof}

It turns out that only $\alpha(k) = +\beta(-k-1)$ results in reflection positivity for the corresponding Laplacian and this is the option we consider below.  The implementation of the covariant derivation from now on will be denoted by $D_\beta$.   We have the following formula:
\begin{equation*}
D_\beta = \pi(U\beta(\K)) - \pi'(\beta(-\K)U).
\end{equation*}
It follows from a simple calculation that $D_\beta$ satisfies:
$$\Theta D_\beta \Theta = -D_\beta.$$  

The following is the second main result of this paper, reflection positivity for a class of quite non-trivial Laplace-type operators coming from covariant derivations.

\begin{theo}\label{refl_pos_covar} Let $m^2$ be a positive number. For all $f\in\mathcal{H}^+$, we have the inequality:
\begin{equation*}
\langle \Theta f, \left((D_\beta)^*D_\beta + m^2\right)^{-1}f\rangle \ge 0. 
\end{equation*}
\end{theo}
\begin{proof}
Since all objects in the inequality are rotationally invariant, the first step is to decompose them into components in each spectral subspace of $U_\varphi$. Intriguingly, the components of the Laplacian are the usual Jacobi operators \cite{T}, which in our case are typically unbounded, self-adjoint, two-step difference operators. Similarly to calculations in \cite{JR} we show that the left-hand side of the inequality can, essentially, be written as a sum of terms over the reflection invariant subset $n+2k=0$ in $\Z^2$. Some of those terms in turn come, by a similar procedure, from manifestly positive inner products as demonstrated by Lemma \ref{pos_bndy}. To contrast it with classical case, integration by parts/Stokes theorem are replaced by more complicated discrete versions.

Recall from \eqref{Hdecomp} that we have the following spectral decomposition $\mathcal{H} \cong \bigoplus_{n\in\Z} \mathcal{H}_n$ and we can naturally identify $\mathcal{H}_n \cong \ell^2(\Z)$ for each $n\in\Z$.  Given $f\in\mathcal{D}$ we have the corresponding decomposition of the Laplacian:
\begin{equation*}
D_\beta^*D_\beta f = \sum_{n\in\Z}U^n (\Delta_n f_n)(\K)
\end{equation*}
where
\begin{equation*}
\begin{aligned}
(\Delta_nf_n)(k) &= \left(|\beta(k+n)|^2 + |\beta(-k)|^2\right)f_n(k) - \overline{\beta}(k+n)\beta(-k-1)f_n(k+1) \\
&\qquad - \beta(k+n-1)\overline{\beta}(-k)f_n(k-1),
\end{aligned}
\end{equation*}
and so we can write:
$$D_\beta^*D_\beta = \bigoplus_{n\in\Z} \Delta_n. $$

Similarly, we also have the decomposition of the reflection operator
$\Theta = \bigoplus_{n\in\Z} \Theta_n,$
where 
$$(\Theta_nf)(k) = f(-k-n)$$
for $f\in\ell^2(\Z)$. It follows that we get $\mathcal{H}^\pm = \bigoplus_{n\in\Z}\mathcal{H}_n^\pm,$ where
$$\mathcal{H}_n^+= \{f\in\ell^2(\Z) : f_n(k) = 0,\textrm{ for }n+2k<0\}.
$$
Consequently, for any positive number $m^2$, we have the following formula:
\begin{equation*}
\langle \Theta f, (D_\beta^*D_\beta + m^2)^{-1}f\rangle = \sum_{n\in\Z}\langle \Theta_nf_n, (\Delta_n + m^2)^{-1}f_n\rangle,
\end{equation*}
which implies that
\begin{equation*}
\langle \Theta f, (D_\beta^*D_\beta + m^2)^{-1}f\rangle\ge 0
\end{equation*}
if and only if for every $n\in\Z$,
\begin{equation*}
\langle \Theta_nf_n, (\Delta_n + m^2)^{-1}f_n\rangle\ge 0.
\end{equation*}

We will now concentrate on proving the above inequality. The first goal is to express the left-hand side of it as boundary terms. This is done in two steps.

Let $f\in \mathcal H^+_n$ and we set:
$$g = (\Delta_n + m^2)^{-1}f.$$ 
Using $\Theta_n : \mathcal{H}_n^+ \to \mathcal{H}_n^-$, we get the following :
\begin{equation*}
\begin{aligned}
&\langle \Theta_nf, (\Delta_n + m^2)^{-1}f\rangle = \sum_{k\in\Z}\overline{f}(-k-n)g(k) =\sum_{n+2k\le0}\overline{f}(-k-n)g(k)\\
&= \sum_{n+2k\le0}\overline{(\Delta_n +m^2)g}(-k-n)g(k)= \sum_{n+2k\le0}(|\beta(k+n)|^2 + |\beta(-k)|^2 + m^2)\overline{g}(-k-n)g(k) \\
&- \sum_{n+2k\le0}\left(\beta(-k)\overline{\beta}(k+n-1)\overline{g}(-k-n+1)g(k)  - \overline{\beta}(-k-1)\beta(k+n)\overline{g}(-k-n-1)g(k)\right).
\end{aligned}
\end{equation*}
Substituting $k\mapsto k-1$ and $k\mapsto k+1$ in the last two terms of the above equation and resumming, we arrive at the following expression:
\begin{equation*}
\begin{aligned}
&\langle \Theta_nf, (\Delta_n + m^2)^{-1}f\rangle = \sum_{n+2k\le0}(|\beta(k+n)|^2 + |\beta(-k)|^2 + m^2)g(k)\overline{g}(-k-n) \\
&\qquad - \sum_{n+2k\le-2}\beta(-k-1)\overline{\beta}(k+n)g(k+1)\overline{g}(-k-n) \\
&\qquad - \sum_{n+2k\le 2}\overline{\beta}(-k)\beta(k+n-1)g(k-1)\overline{g}(-k-n).
\end{aligned}
\end{equation*}
There are two cases that will be addressed separately: $n$ even and $n$ odd.  If $n$ is even we have
\begin{equation*}
\begin{aligned}
&\langle \Theta_nf, (\Delta_n + m^2)^{-1}f\rangle = \sum_{n+2k\le-2}(\Delta_n+m^2)g(k)\overline{g}(-k-n) + \left(2\left|\beta\left(\frac{n}{2}\right)\right|^2 + m^2\right)\left|g\left(-\frac{n}{2}\right)\right|^2 \\
&- \overline{\beta}\left(\frac{n}{2}\right)\beta\left(\frac{n}{2}-1\right)\overline{g}\left(-\frac{n}{2}\right)g\left(-\frac{n}{2}-1\right)- \beta\left(\frac{n}{2}\right)\overline{\beta}\left(\frac{n}{2}-1\right)g\left(-\frac{n}{2}\right)\overline{g}\left(-\frac{n}{2}-1\right)\\
&= \left(2\left|\beta\left(\frac{n}{2}\right)\right|^2 + m^2\right)\left|g\left(-\frac{n}{2}\right)\right|^2 - (x + \overline{x}),
\end{aligned}
\end{equation*}
where we used the fact that $f=(\Delta_n +m^2)g$, and $f(k)=0$ for $n+2k<0$. We denoted the two boundary conjugate terms by $x$ and $\overline{x}$ respectively, with
\begin{equation}\label{xdefref}
x= \overline{\beta}\left(\frac{n}{2}\right)\beta\left(\frac{n}{2}-1\right)\overline{g}\left(-\frac{n}{2}\right)g\left(-\frac{n}{2}-1\right).
\end{equation}

If $n$ is odd, a similar calculation to the above results in the following formula:
\begin{equation*}
\langle \Theta_nf, (\Delta_n + m^2)^{-1}f\rangle = \left|\beta\left(\frac{n-1}{2}\right)\right|^2\left(\left|g\left(\frac{1-n}{2}\right)\right|^2 - \left|g\left(\frac{-n-1}{2}\right)\right|^2\right).
\end{equation*}

To prove positivity we need the following observation that relates boundary terms of the type we encountered above with manifestly positive sums over bigger regions.

\begin{lem}\label{pos_bndy}
For $f\in\mathcal{H}_n^+$ let $g = (\Delta_n + m^2)^{-1}f$ and consider two nonnegative sums:
\begin{equation*}
\begin{aligned}
&\Sigma_1 = \sum_{n+2k<0}|\beta(k+n)g(k) - \beta(-k-1)g(k+1)|^2 + m^2\sum_{n+2k<0}|g(k)|^2 \\
&\Sigma_2 = \sum_{n+2k<0}|\beta(-k)g(k) - \beta(k+n-1)g(k-1)|^2 + m^2\sum_{n+2k<0}|g(k)|^2.
\end{aligned}
\end{equation*}
Then we have the following formulas:
\begin{equation*}
\Sigma_1 =\left\{ 
\begin{aligned}
& \left|\beta\left(\frac{n}{2}\right)\right|^2\left|g\left(-\frac{n}{2}\right)\right|^2 -\beta\left(\frac{n}{2}-1\right)\overline{\beta}\left(\frac{n}{2}\right)g\left(-\frac{n}{2}-1\right)\overline{g}\left(-\frac{n}{2}\right) && n\textrm{ even} \\
& \left|\beta\left(\frac{n-1}{2}\right)\right|^2\left|g\left(\frac{1-n}{2}\right)\right|^2 - \left|\beta\left(\frac{n-1}{2}\right)\right|^2\overline{g}\left(\frac{1-n}{2}\right)g\left(\frac{-n-1}{2}\right) && n\textrm{ odd}
\end{aligned}\right.
\end{equation*}
and
\begin{equation*}
\Sigma_2 = \left\{
\begin{aligned}
& -\left|\beta\left(\frac{n}{2}-1\right)\right|^2\left|g\left(-\frac{n}{2}-1\right)\right|^2 +\beta\left(-\frac{n}{2}\right)\overline{\beta}\left(\frac{n}{2}-1\right)\overline{g}\left(-\frac{n}{2}-1\right)g\left(-\frac{n}{2}\right) && n\textrm{ even} \\
& -\left|\beta\left(\frac{n-1}{2}\right)\right|^2\left|g\left(\frac{-n-1}{2}\right)\right|^2 + \left|\beta\left(\frac{n-1}{2}\right)\right|^2\overline{g}\left(\frac{-n-1}{2}\right)g\left(\frac{1-n}{2}\right) && n\textrm{ odd}
\end{aligned}\right.
\end{equation*}
\end{lem}
\begin{proof}
We prove only the formula for $\Sigma_1$ in the statement above as the proof of the formula for $\Sigma_2$ is virtually identical.  First we split the summation as follows:
\begin{equation*}
\begin{aligned}
\Sigma_1 &= \sum_{n+2k<0}|\beta(k+n)|^2\overline{g}(k)g(k) + \sum_{n+2k<0}|\beta(-k-1)|^2\overline{g}(k+1)g(k+1)+ m^2\sum_{n+2k<0}\overline{g}(k)g(k)  \\
&- \sum_{n+2k<0}\overline{\beta}(k+n)\beta(-k-1)\overline{g}(k)g(k+1) - \sum_{n+2k<0}\beta(k+n)\overline{\beta}(-k-1)g(k)\overline{g}(k+1). \\
\end{aligned}
\end{equation*}
By substituting $k\mapsto k-1$ in the second and fifth sum and resumming we get
\begin{equation*}
\begin{aligned}
\Sigma_1 &= \sum_{n+2k<0}|\beta(k+n)|^2\overline{g}(k)g(k) + \sum_{n+2k<2}|\beta(-k)|^2\overline{g}(k)g(k)+ m^2\sum_{n+2k<0}\overline{g}(k)g(k) \\
& - \sum_{n+2k<0}\overline{\beta}(k+n)\beta(-k-1)\overline{g}(k)g(k+1) - \sum_{n+2k<2}\beta(k+n-1)\overline{\beta}(-k)g(k-1)\overline{g}(k). \\
\end{aligned}
\end{equation*}
Rearranging the terms we obtain:
\begin{equation*}
\begin{aligned}
\Sigma_1&= \sum_{n+2k<0}(\Delta_n +m^2)g(k)\overline{g}(k) + \sum_{n+2k=0}|\beta(-k)|^2\overline{g}(k)g(k) + \sum_{n+2k=1}|\beta(-k)|^2\overline{g}(k)g(k) \\
& - \sum_{n+2k=1}\beta(k+n-1)\overline{\beta}(-k)\overline{g}(k)g(k-1) - \sum_{n+2k=0}\beta(k+n-1)\overline{\beta}(-k)\overline{g}(k)g(k-1).
\end{aligned}
\end{equation*}
Since $g = (\Delta_n + m^2)^{-1}f$ and $f\in \mathcal{H}_n^+$, we have $(\Delta_n + m^2)g(k) = f(k) = 0$ for $n+2k<0$.  Therefore, $\Sigma_1$ is equal to the four remaining boundary terms in the above equation and depending on whether $n$ is even or odd, the result follows.
\end{proof}

Notice that the formula for $\Sigma_1$ for even $n$ in Lemma \ref{pos_bndy} implies that $x$ from \eqref{xdefref} is real and so we must have $x = \overline{x}$.  Therefore, if $n$ is even we have:
\begin{equation*}
\begin{aligned} 
\langle \Theta_nf, (\Delta_n + m^2)^{-1}f\rangle &=
2\left(\left|\beta\left(\frac{n}{2}\right)\right|^2 -x\right)\left|g\left(-\frac{n}{2}\right)\right|^2 + m^2\sum_{n+2k=0}|g(k)|^2\\
&=2\Sigma_1 + 2m^2\sum_{n+2k=0}|g(k)|^2.\\
\end{aligned} 
\end{equation*}

If we set
\begin{equation*}
y = \left|\beta\left(\frac{n-1}{2}\right)\right|^2\overline{g}\left(\frac{1-n}{2}\right)g\left(\frac{-n-1}{2}\right),
\end{equation*} 
then by either the formula for $\Sigma_1$ or for $\Sigma_2$ for odd $n$ in Lemma \ref{pos_bndy}, we again have that the quantity $y$ is real and consequently we have $y-\overline{y} = 0$.  Thus, we obtain the following expression for odd $n$:
\begin{equation*}
\langle \Theta_nf, (\Delta_n + m^2)^{-1}f\rangle = \left|\beta\left(\frac{n-1}{2}\right)\right|^2\left(\left|g\left(\frac{1-n}{2}\right)\right|^2 - \left|g\left(\frac{-n-1}{2}\right)\right|^2\right) + y - \overline{y}=\Sigma_1 + \Sigma_2.
\end{equation*}
Putting both cases together yields the following formulas:
\begin{equation*}
\langle \Theta_nf, (\Delta_n + m^2)^{-1}f\rangle=\left\{
\begin{aligned} 
&2\Sigma_1 + 2m^2\sum_{n+2k=0}|g(k)|^2 &&n\textrm{ even}\\
&\Sigma_1 + \Sigma_2 &&n\textrm{ odd,}
\end{aligned}\right.
\end{equation*}
which are manifestly nonnegative. This completes the proof.
\end{proof}

\end{document}